\documentclass{article}
\usepackage{hyperref}
\usepackage{amsthm}
\usepackage{amsmath,amssymb}
\usepackage{cleveref}

\usepackage[sort, numbers]{natbib}

 \usepackage{xparse}
\NewDocumentCommand{\DIV}{om}{%
  \IfValueT{#1}{\setcounter{#2}{\numexpr#1-1\relax}}%
  \csname #2\endcsname
}
\makeatletter
\renewcommand\section{\@startsection {section}{1}{\z@}%
                                   {-3.5ex \@plus -1ex \@minus -.2ex}%
                                   {2.3ex \@plus.2ex}%
                                   {\normalfont\scshape}}
\usepackage{bm}
\newcommand{\bu}{\bm{u}}

\newcommand\bE{\mathbb{E}}
\newcommand\bP{\mathbb{P}}
\newtheorem*{theorem}{Theorem}

\author{Sina Baghal\footnote{Email: srezazadehbaghal@uwaterloo.ca}}
\date{}
 
\title{A matrix concentration inequality for products}
\begin{document}
\maketitle
\begin{abstract}
Abstract. We present a \textit{non-asymptotic concentration inequality} for the random matrix product 
    \begin{equation}\label{eq:Zn}
        Z_n = \left(I_d-\alpha X_n\right)\left(I_d-\alpha X_{n-1}\right)\cdots \left(I_d-\alpha X_1\right),
    \end{equation}
    where $\left\{X_k \right\}_{k=1}^{+\infty}$ is a sequence of bounded independent random positive semidefinite matrices with common expectation $\mathbb{E}\left[X_k\right]=\Sigma$. Under these assumptions, we show that, for small enough positive $\alpha$, $Z_n$ satisfies the concentration inequality 
    \begin{equation}\label{eq:CTbound}
     \mathbb{P}\left(\left\Vert Z_n-\mathbb{E}\left[Z_n\right]\right\Vert \geq t\right) \leq 2d^2\cdot\exp\left(\frac{-t^2}{\alpha \sigma^2} \right) \quad \text{for all } t\geq 0,
\end{equation}
where $\sigma^2$ denotes a variance parameter.
\end{abstract}
\begin{center}
\item \section*{1. Motivation}
\end{center}
Products of random matrices appear as building blocks for many stochastic iterative algorithms, \textit{e.g.} \cite{Oja1982, strohmer2007randomized}. While non-asymptotic bounds of averages of these matrices are well developed, \textit{e.g.} \cite{Tropp_2011, wainwright_2019}, the analogous bounds of their products are much harder to understand due to the non-commutative nature of matrix multiplication.  As such, efforts to understand bounds of this type have become an active area of research \textit{e.g.} \cite{HENRIKSEN202081, huang2020matrix, kathuria2020concentration}. 
\begin{center}
\item \section*{2. Contribution}
\end{center}
 In this note, we provide a non-asymptotic concentration bound \eqref{eq:CTbound} for the random matrix product $Z_n$ \eqref{eq:Zn}. These instances appear, for example, in the stochastic gradient descent algorithm applied to the linear least squares problem. We remark that bound \eqref{eq:CTbound} will be of special interest when $X_k$ is almost surely low rank for all $k$. In this event, almost all eigenvalues of each factor in the matrix product $Z_n$ are equal to 1 whereas $\bE[Z_n]$ has an exponentially decaying operator norm. (Note: Without loss of generality, we can assume that $\Sigma$ is positive definite.) Hence, it is interesting to observe that $Z_n$ concentrates around its mean with overwhelming probability as in \eqref{eq:CTbound}, especially in the case where $X_k$'s are almost surely low rank matrices.
 \begin{center}
   \item \section*{3. Related Work}  
 \end{center}
 In \cite{huang2020matrix}, using the uniform smoothness property of the Schatten $p$-norm, the authors have studied non-asymptotic bounds for the products of random matrices, in particular, random contractions \cite[Theorem 7.1]{huang2020matrix}. To apply their result to the matrix product \eqref{eq:Zn}, we will need to make some further assumptions. First, we need to assume some bound involving $\left\vert I-\alpha X_k\right\vert$ since the Araki-Lieb-Thirring inequality \cite[IX.2.11]{Bhatia} is used in their analysis. Second, we need to assume a lower bound $t^2\geq c\alpha^2\sum_{k=1}^n \Vert X_k-\Sigma\Vert^2$ which may grow linearly in $n$. This will be problematic particularly since we are only interested in the case where $t\leq 1$.

 On the other hand, compared to our result, the bound in \cite[Theorem 7.1]{huang2020matrix} has a weaker dependency on the dimension $d$  and, more importantly, it works in a broader variety of instances. For example, one can use their bound when in \eqref{eq:Zn}, instead of $I-\alpha X_k$, we consider the factors $I-\alpha_k X_k$ with $\alpha_k$ decaying at a proper rate.
 \begin{center} 
\item \section*{4. Concentration bound}
\end{center}
In this section, we prove our result \eqref{eq:CTbound}. The proof proceeds by constructing a martingale sequence satisfying bounded differences and then applying Azuma's inequality. We assume that the positive semidefinite random matrices $X_k$ in \eqref{eq:Zn} are drawn independently and they satisfy $\bE\left[X_k\right]=\Sigma$ for all $k$. In addition, we suppose that $X_k$ are uniformly bounded in the operator norm, meaning that there exists $r>0$ such that
\begin{equation*}
    \Vert X_k\Vert\leq r \quad \text{almost surely.}
\end{equation*}
Let $\bu_1,\cdots,\bu_d$ denote the eigenvectors of $\Sigma$ and $\lambda_1\geq \cdots \geq \lambda_d\geq  0$ denote the corresponding eigenvalues. For each $i\in [d]:=\{1,\cdots,d\}$, define $c_i$ to be the infimum over all positive reals for which
\begin{equation*}
     \left\Vert \left(X_k-\lambda_i I\right)\bu_i \right\Vert \leq c_i\lambda_i \quad \text{almost surely}.
\end{equation*}
Note that $c_i<+\infty$ almost surely as $c_i\leq 1+\frac{r}{\lambda_i}$ and also, because $X_k$ is positive semidefinite, $c_i=0$ whenever $\lambda_i=0$. We will use the following parameter to measure the amount of variation in $X_k$
\begin{equation*}\label{eq:V}
    \sigma^2:= \frac{4d}{3}\sum_{i=1}^{d}c_i^2\lambda_i.
\end{equation*}
 \begin{theorem}\label{prop1}
Suppose that $\alpha \in (0,\frac{1}{2r})$. Then the following concentration inequality holds. 
\begin{equation}\label{eq:norm_bound}
     \bP\left(\left\Vert Z_n-\bE\left[Z_n\right]\right\Vert \geq t\right) \leq 2d^2\cdot\exp\left(\frac{-t^2}{\alpha \sigma^2} \right).
\end{equation}
\end{theorem}
\begin{proof}
Without loss of generality, we can assume that $c_i,\lambda_i>0$ for all $i\in [d]$. We will first show that, for any $i,j\in [d]$, the following bound holds for all $t\geq 0$:
\begin{equation}\label{eq:zkij}
    \bP\left(\left\vert \bu_i^TZ_n\bu_j-\bE\left[\bu_i^TZ_n\bu_j\right]\right\vert \geq t\sqrt{\tfrac{4\lambda_i}{3}}\cdot c_i\right) \leq 2\exp\left(\frac{-t^2}{\alpha} \right).
\end{equation}
Set $Z_0=I_d$. Then we note that for all $k\geq 0$, 
\begin{equation*}
    \bE\left[\bu_i^TZ_k\bu_j\right] =  \bu_i^T\bE\left[Z_k\right]\bu_j=\bu_i^T\left(I-\alpha\Sigma\right)^k\bu_j= \left(1-\alpha \lambda_i\right)^k\cdot \delta_{i,j},
\end{equation*}
where $\delta_{i,j}$ stands for Kronecker delta.  For notational convenience, let us denote $z_k:=\bu_i^TZ_k\bu_j$. We have that 
\begin{align}\label{eq:mukmuk-1}
        \bE\left[z_k|X_{k-1},\cdots,X_{1}\right]&= \bu_{i}^T\bE\left[I-\alpha X_k \right]Z_{k-1}\bu_{j}=\left(1-\alpha \lambda_i\right)z_{k-1}.
    \end{align}
Denote $q_i:=1-\alpha \lambda_i$ and define the random variable $Y_k:=q_i^{-k}\cdot z_k$. Dividing both sides of \eqref{eq:mukmuk-1} by $q_i^k$, we obtain that $\bE\left[Y_k | X_{k-1},\cdots,X_1\right]=Y_{k-1}$. Thus, $\{Y_k\}_{k=1}^{+\infty}$ is a martingale with respect to $\{X_k\}_{k=1}^{+\infty}$. We observe that for all $k\geq 1$
\begin{equation*}
\begin{aligned}
     q_i^k\cdot \vert Y_k-Y_{k-1}\vert &= \vert z_k-q_i\cdot z_{k-1} \vert = \alpha \left\vert \bu_i^T\left(X_k-\lambda_i I\right)Z_{k-1}\bu_j\right\vert \leq \alpha c_i\lambda_i,
\end{aligned}
   \end{equation*}
   where the assumption $\alpha r\leq \frac{1}{2}$ yielded the bound $\Vert Z_{k-1}\Vert \leq 1$ a.s.   Thus, by Azuma's inequality, see \textit{e.g.} \cite{alon}, we have that for any $\epsilon\geq 0$
   \begin{equation}\label{eq:blah_blah}
   \begin{aligned}
       \bP\left(\vert z_n-\bE\left[z_n\right]\vert \geq \epsilon\right)&=\bP\left(\vert Y_n-Y_0\vert \geq \epsilon\cdot q_i^{-n}\right) \\&\leq  2\exp\left(\frac{-\epsilon^2}{2\alpha^2\lambda_i^2c_i^2\sum_{k=0}^{n-1} q_i^{2k}}\right).
       \end{aligned}
   \end{equation}
Note that by Jensen's inequality  
\begin{equation}\label{eq:jensen}
    \lambda_i \leq \Vert \Sigma\Vert = \Vert \bE\left[X_k\right]\Vert\leq \bE\left[\Vert X_k\Vert\right] \leq r.
\end{equation}
Therefore, by \eqref{eq:jensen} and since $\alpha \in (0,\frac{1}{2r})$, we obtain that $\sum_{k=0}^{n-1} q_i^{2k}\leq \frac{2}{3(1-q_i)}$. Plugging this bound into the right-hand side of \eqref{eq:blah_blah} and letting $\epsilon=t\sqrt{\tfrac{4\lambda_i}{3}}\cdot c_i$, we will obtain \eqref{eq:zkij}.  Finally, in order to see \eqref{eq:norm_bound}, we observe that by \eqref{eq:zkij}, with probability exceeding $1-2d^2\cdot \exp\left(-\frac{t^2}{\alpha}\right)$, it holds that
\begin{equation*}
    \left\Vert Z_n-\bE\left[Z_n\right]\right\Vert_F^2 \leq t^2\cdot \frac{4d}{3}\sum_{i=1}^{d}c_i^2\lambda_i,
\end{equation*}
where $\left\Vert Z_n-\bE\left[Z_n\right]\right\Vert_F$ is the Frobenius norm. Therefore, 
\begin{equation*}
         \bP\left(\left\Vert Z_n-\bE\left[Z_n\right]\right\Vert_F \geq t\cdot \sigma\right) \leq 2d^2\cdot\exp\left(\frac{-t^2}{\alpha} \right).
\end{equation*}
The result immediately follows since $\left\Vert Z_n-\bE\left[Z_n\right]\right\Vert\leq \left\Vert Z_n-\bE\left[Z_n\right]\right\Vert_F$.
\end{proof}

\end{document}